\documentclass[12pt]{amsart}
\usepackage{amssymb,latexsym}
\usepackage{enumerate}
\usepackage[T1]{fontenc}
\makeatletter
\@namedef{subjclassname@2010}{%
  \textup{2010} Mathematics Subject Classification}
\makeatother

\newtheorem{tw}{Theorem}[section]

\newtheorem{lm}[tw]{Lemma}
\newtheorem{prop}[tw]{Proposition}

\theoremstyle{definition}
\newtheorem{df}[tw]{Definition}

\numberwithin{equation}{section}

\newcommand{\mb}{\mathbb}
\newcommand{\mc}{\mathcal}
\newcommand{\mf}{\mathfrak}

\newcommand{\s}{\subset}

\begin{document}

\baselineskip=17pt

\title[the relative extremal function]{A remark on the relative extremal function}

\author{Arkadiusz Lewandowski}
\thanks{Project operated within the Foundation for Polish Science
IPP Programme ``Geometry and Topology in Physical Models''
co-financed by the EU European Regional Development Fund,
Operational Program Innovative Economy 2007-2013.}
\address{Institute of Mathematics\\ Faculty of Mathematics and Computer Science\\ Jagiellonian University\\ {\L}ojasiewicza 6,
30-348 Kraków, Poland}
\email{arkadiuslewandowski@gmail.com, Arkadiusz.Lewandowski@im.uj.edu.pl}

\date{}

\begin{abstract}
The main result of this paper is some ``annulus'' formula for the relative extremal function in the context of Stein spaces (Theorem \ref{pierscien}). Our result may be useful in the theory of the extension of separately holomorphic functions on generalized $(N,k)$-crosses lying in the product of Stein manifolds (Theorem \ref{main}).
\end{abstract}



\maketitle

\section{Introduction}
In \cite{J3} Jarnicki and Pflug proved a Hartogs type extension theorem for $(N,k)$-crosses lying in the product of Riemann domains of holomorphy over $\mb{C}^n$, which is a generalization of the classical cross theorem (see, for example \cite{A1}). The key role in their proof is played by some ``annulus'' formula for the relative extremal function. The aim of the present paper is to extend that formula to the situation, where instead of the Riemann domains of holomorphy over $\mb{C}^n$ we consider Stein spaces. Namely, we shall prove the following (for the necessary definitions see Section \ref{sec2}).
\begin{tw}
Let $D\s\s X,$ where for the couple $(D,X)$ at least one of the following two conditions is satisfied:
\begin{enumerate}[\upshape (A)]
\item $D$ is an irreducible, locally irreducible weakly parabolic Stein space with some potential $g$ and $X$ is a Stein space, 
\item $D$ is a Stein manifold and $X$ is a Josefson manifold. 
\end{enumerate}
Let $A\s D$ be nonpluripolar. Define
$$\Delta(r):=\{z\in D:h_{A,D}^{\star}(z)<r\}, \quad r\in(0,1].$$
Then for $0<r<s\leq 1$ we have
$$
h^{\star}_{\Delta(r),\Delta(s)}=\max \Big\{0,\frac{h^{\star}_{A,D}-r}{s-r}\Big\} \quad\text{ on $\Delta(s)$}.
$$ 
\label{pierscien}
\end{tw}
Note that the class of Josefson manifolds (i.e. those complex manifolds, for which any locally pluripolar set is globally pluripolar) is essentialy wider than the class of Stein manifolds (see Theorem 5.3 in \cite{B1}).
The above result will also allow us (see Section \ref{sec5}) to prove the formula for the relatively extremal function of the envelope of $(N,k-1)$-cross with
respect to the envelope of $(N,k)$-cross (Theorem \ref{Obwiednia w obwiedni}; cf. \cite{J3}). Finally we use our main result to give a new Hartogs type extension theorem for the generalized $(N,k)$-crosses (introduced in \cite{L1}) in the context of Stein manifolds. In the author's intention the present paper is a step towards the extenstion of separately holomorphic functions on the generalized $(N,k)$-crosses in the context of arbitrary complex manifolds, or even complex spaces.\\
\indent The paper was written during the author's stay at the Carl von Ossietzky Universit\"{a}t Oldenburg. The author would like to express his gratitude to Professor Peter Pflug for his constant help and inspiring discussions.   
\section{Prerequisites}
\label{sec2}
This section contains some definitions and results which will be needed in the sequel.\\ 
\indent We assume that any considered here complex space $X$ is reduced, has a~countable basis of topology and is of pure dimension. If $X$ is a complex space, then any $x\in X$ possesses an open neighborhood $U$ and a biholomorphic mapping $\varphi$ from $U$ to some subvariety $B$ of a domain $V\s\mb{C}^n.$ The 4-tuple $(U,\varphi,B,V)$ will be called a \emph{chart of $X$}. Also, we will use the notation Reg$X$ for the set of all regular points of $X$ and Sing$X$ for the set of all singular points of $X$ (see \cite{Lo}, Chapter V). In the present paper $\mc{PLP}(X)$ stands for the family of all (locally) pluripolar subsets of $X$ and $\mc{O}(X)$ is the space of all holomorphic functions on $X$. Finally, we assume throughout the paper that any appearing complex manifold is countable at infinity.
\begin{df}\rm
Let $X$ be a complex space. A function $u:X\rightarrow[-\infty,\infty)$,\ $u \not\equiv-\infty$~on~irreducible componnents of $X$, is called \emph{plurisubharmonic} (written $u\in\mc{PSH}(X)$) if for any $x\in X$ there are a chart $(U,\varphi,B,V)$ with $x\in U$ and a function $\psi\in\mc{PSH}(V)$ with $\psi\circ \varphi=u|_U.$  
\end{df}
The following result plays a central role in the theory of plurisubharmonic functions on complex spaces.
\begin{tw}[\cite{FN1}]
An upper semicontinuous function $u:X\rightarrow[-\infty,\infty)$ is plurisubharmonic on $X$ iff for any function $f\in\mc{O}(\mb{D},X),$ the function $u\circ f$ is subharmonic on $\mb{D}.$ Here $\mb{D}$ means the unit disc in the complex plane.  \label{disc}
\end{tw}
Note that the above result immediately implies the following ``basic'' properties of plurisubharmonic functions.
\begin{prop}[cf. \cite{Sm}]
Let $X$ be a complex space.
\begin{enumerate}[\upshape (a)]
\item Let $(u_n)_{n\in\mb{N}}\s\mc{PSH}(X)$. If $u:=\sup\{u_n\}$ is upper semicontinuous and $u<\infty,$ then it is also plurisubharmonic.
\item Let $(u_n)_{n\in\mb{N}}\s\mc{PSH}(X)$. If $(u_n)_{n\in\mb{N}}$ is decreasing and $u:=\inf\{u_n\}$ is not identically $-\infty$ on any irreducible component of $X,$ then it is also plurisubharmonic.
\item Let $(u_n)_{n\in\mb{N}}\s\mc{PSH}(X)$. If $(u_n)_{n\in\mb{N}}$ converges uniformly, then its limit is also plurisubharmonic.
\item Let $Y\s X$ be open. Let $v\in\mc{PSH}(Y),u\in\mc{PSH}(X)$ and such that
$$
\limsup\limits_{Y\ni x\rightarrow x_0}v(x)\leq u(x_0),\quad x_0\in\partial Y.
$$
Put
\begin{displaymath}
\tilde{u}(x):=\begin{cases}
\max\{v(x),u(x)\},& x\in Y,\\
u(x),&x\in X\setminus Y.
\end{cases}
\end{displaymath}
Then $\tilde{u}$ is plurisubharmonic on $X.$
\end{enumerate}
\label{glue}
\end{prop}
\begin{df}[\cite{GR}, Chapter VII, Section A]\rm
Let $X$ be a complex space and let $K\s X$ be compact. The \emph{holomorphically convex hull of $K$ in $X$} is defined as
\begin{displaymath}
\hat{K}_{X}:=\{x\in X:|f(x)|\leq||f||_{K},f\in\mc{O}(X)\}.
\end{displaymath}
We say that $K$ is \emph{holomorphically convex}, if $K=\hat{K}_X.$ A complex space $X$ is called \emph{holomorphically convex}, if for any compact set $K\s X,$ the set $\hat{K}_X$ is also compact.
\end{df}
\begin{prop}[\cite{FN1}]
Let $X$ be a Stein space (see \cite{GR}, Chapter VII, Section A, Definition 2) and let $u\in\mc{PSH}(X).$ Then for any real number $c,$ the set $Y:=\{x\in X:u(x)<c\}$ is Runge in $X$ (that is, for any compact set $K\s Y$, the set $\hat{K}_X\cap Y$ is compact, see \cite{N2}), and, in particular, it is also a Stein space.
\label{Runge}
\end{prop}
\begin{tw}[\cite{N1}]
Let $X$ be a Stein space. Then there exists a real analytic, strongly plurisubharmonic exhaustion function on $X$.
\label{exh}
\end{tw}
Note that the real analyticity on a complex space $X$ is defined in a similar way like the plurisubharmonicity. A function $f$ on $X$ is real analytic, if for any $x\in X$ there are a chart $(U,\varphi,B,V)$ with $x\in U$ and a real analytic function $g$ on $V$ with $g\circ\varphi=f|_U$ (see \cite{N1}).\\
\indent For a function $\psi$ as in Theorem \ref{exh} and for any real number $c$ denote by $\Omega_c(\psi)$ the sublevel set $\{x\in X:\psi(x)<c\}.$
\begin{df} [\cite{A1}]
Let $X$ be a complex space. Let $D\s X$ be open and let $A\s D.$ Define the \emph{relative extremal function of $A$ with respect to $D$} as the standard upper semicontinuous regularization $h^{\star}_{A,D}$ of the function 
\begin{displaymath}
h_{A,D}:=\sup\{u:u\in\mc{PSH}(D),u\leq 1,u|_A\leq 0\}.
\end{displaymath}
For an open set $Y\s X$ we put $h_{A,Y}:=h_{A\cap Y,Y},h^{\star}_{A,Y}:=h^{\star}_{A\cap Y,Y}.$
\label{extremal}
\end{df}
\begin{df}[\cite{A1}]
We say that a set $A\s X$ is \emph{pluriregular~at~a~point} $a\in\overline{A}$ if $h^{\star}_{A,U}(a)=0$ for any open neighborhood $U$ of $a.$ Define
\begin{displaymath}
A^{\star}:=\{a\in\overline{A}:A\text{ is pluriregular at $a$}\}.
\end{displaymath}
We say that $A$ is \emph{locally pluriregular} if~$A~\neq~\varnothing$ and $A$ is pluriregular at each of its points, i.e. $\varnothing\neq A\s A^{\star}.$
\label{df3}
\end{df}
\begin{tw}[\cite{AH}]
Let $X$ be an irreducible Stein space and let $D\s X$ be a domain. Then the following conditions are equivalent.
\begin{enumerate}[\upshape (a)]
\item For any set $P\in\mc{PLP}(D)$ and any $A\s D$ we have $h^{\star}_{A\cup P,D}\equiv h^{\star}_{A,D}$
\item For any set $P\in\mc{PLP}(D)$ there exists a $u\in\mc{PSH}(D),u\leq 0$ and nonconstant, such that $P\s\{u=-\infty\}.$
\end{enumerate}
\label{alehyane}
\end{tw}
\begin{lm}[\cite{J2}, Propostion 3.2.27, Lemma 6.1.1] Let $X$ a complex space, $A\s X$ locally pluriregular, and $\varepsilon\in(0,1).$ Put
$$
X_{\varepsilon}:=\{z\in X:h^{\star}_{A,X}(z)<1-\varepsilon\}.
$$
Then for any connected component $D$ of $X_{\varepsilon}$ we have
\begin{enumerate}[\upshape (a)]
\item $A\cap D\neq\varnothing.$
\item $h^{\star}_{A\cap D,D}(z)=\frac{h^{\star}_{A,X}(z)}{1-\varepsilon},z\in D.$
\end{enumerate}
\label{sublevel}
\end{lm}
\begin{proof}
The proof goes along the same lines as the proof of Proposition 3.2.27 from \cite{J2}. We only need to use Proposition \ref{glue}(d) instead of Proposition 2.3.6 from \cite{J2}.
\end{proof}
\begin{prop}[cf. Proposition 3.2.23 in \cite{J2}]
Let $X_k\nearrow X\s\s Y,$ where $X$ is a Stein space and $Y$ is a complex space for which Josefson's theorem is valid, let $A_k\s X_k,A_k\nearrow A.$ Then $h^{\star}_{A_k,X_k}\searrow h^{\star}_{A,X}.$ 
\label{3.2.23}
\end{prop}
\begin{proof}
The proof is the same as the one of Proposition 3.2.23 in \cite{J2}; only, we use Lemma 2.2 from \cite{AH} instead of Corollary 3.2.12.
\end{proof}
\begin{prop}[cf. Proposition 3.2.15 in \cite{J2}]
Let $Y$ be an irreducible Stein space. Let $X=\Omega_c(\psi)$ with some $c\in\mb{R}$ and $\psi$ as in Theorem \ref{exh} for $Y,$ and let $A\s X$.
Then for any $\varepsilon\in(0,1)$ we have
\begin{displaymath}
\frac{h^{\star}_{A,X}-\varepsilon}{1-\varepsilon}\leq h^{\star}_{\Delta(\varepsilon),X}\leq h^{\star}_{A,X}.
\end{displaymath}
\label{3.2.15}
\end{prop}
\begin{proof}
The proof is the same as the proof of Proposition 3.2.15 from \cite{J2}. We only need to use Lemma 2.6 and Theorem 2.1 from \cite{AH} instead of Proposition 3.2.2 and Proposition 3.2.11, respectively.
\end{proof}
\begin{prop}[cf. Proposition 4.5.2 in \cite{K1}]
Let $Y$ be an irreducible Stein space. Let $X=\Omega_c(\psi)$ with some $c\in\mb{R}$ and $\psi$ as in Theorem \ref{exh} for $Y,$ and let $A\s X$ be relatively compact. Then, for any point $x_0\in\partial X$ we have$\lim\limits_{X\ni x\rightarrow x_0}h_{A,X}=1.$
\label{4.5.2}
\end{prop}
\begin{proof}
The proof is as the one given in \cite{K1}, since it depends only on the existence of an exhaustion function for $X$.
\end{proof}
\begin{prop}[cf. Proposition 3.2.24 in \cite{J2}]
Let $X$ be a Stein space and let $(K_j)_{j\in\mb{N}}$ be a decreasing sequence of compact subsets of $X$ with $\bigcap\limits_{j\in\mb{N}}K_j=K.$ Then $h_{K_j,X}\nearrow h_{K,X}.$
\label{3.2.24}
\end{prop}
\begin{proof}
The proof may be rewritten verbatim from \cite{J2}.
\end{proof}
\indent The complex Monge-Amp{\`e}re operator $(dd^cu)^n$ for a locally bounded function $u\in\mc{PSH}(X)$ is defined in a standard way on Reg$X$ (\cite{BT1}) and it is extended ``by zero'' through Sing$X$ (for details and the further theory see \cite{B1}).\\
Note that (see \cite{A1}) if $D$ is hyperconvex (i.e. there exists a plurisubharmonic negative function $\eta$ such that for any $c<0$ the set $\{z\in D:\eta(z)<c\}$ is relatively compact in $D$) and $A$ is compact, then
$(dd^ch^{\star}_{A,D})^n=0$ on $D\setminus A.$
\begin{tw}[Comparison theorem, see \cite{B1}]
Let $X$ be a complex space and let $u,v\in\mc{PSH}(X)\cap L^{\infty}_{loc}(X)$ be such that the set $\{u\leq v\}$ is relatively compact in $X.$ Then
\begin{displaymath}
\int_{\{u<v\}}(dd^cu)^n\geq\int_{\{u<v\}}(dd^cv)^n
\end{displaymath}\label{comp}
\end{tw}
\begin{tw}[cf. Theorem 3.2.32 in \cite{J2}, Corollary 3.7.4 in \cite{K1}]
Let $\Omega\s\s D\s\s X,$ where $X$ is a Stein space, $D=\Omega_c(\psi)$ with some $c\in\mb{R}$ and $\psi$ as in Theorem \ref{exh} for $X$, and $\Omega$ is an open set. Let $u,v\in\mc{PSH}(\Omega)\cap L^{\infty}(\Omega)$ such that $(dd^cv)^n\geq(dd^cu)^n$ on $\Omega$ and
$$
\liminf_{\Omega\ni z\rightarrow z_0}(u(z)-v(z))\geq 0, \quad z_0\in\partial\Omega.
$$
Then $u\geq v$ on $\Omega.$
\label{domination}
\end{tw}
\begin{proof}
Observe that $\eta:=\psi-c<0$ is a real analytic strongly plurisubharmonic exhaustion function for $D$. Then there is some $C<0$ satisfying $\overline{\Omega}\s\{\eta<C\}.$ If now $\{u<v\}\neq\varnothing,$ then also $S:=\{u<v+\varepsilon\eta\}$ is nonempty for some $\varepsilon>0.$ Moreover, the set $S\cap\rm{Reg}D$ is of positive Lebesgue measure. Also, $\{u\leq v+\varepsilon\eta\}$ has to be relatively compact in $\Omega.$ Hence we get
$$
\int_S(dd^cu)^n\geq\int_S(dd^c(v+\varepsilon\eta))^n\geq\int_S(dd^cv)^n+\varepsilon^n\int_S(dd^c\eta)^n>\int_S(dd^cv)^n,
$$
a contradiction (note that the first inequality above is the consequence of Theorem \ref{comp}).
\end{proof}
\begin{tw}[cf. Corollary 3.2.33 in \cite{J2}]
Let $X$ be a Stein space, $D=\Omega_c(\psi)$ with some $c\in\mb{R}$ and $\psi$ as in Theorem \ref{exh} for $X$, $K\s\s D$ compact, and let $U\s D\setminus K$ be open. Assume that $h^{\star}_{K,D}$ is continuous and let $u\in\mc{PSH}(U)\cap L^{\infty}(U),u\leq 1$ and such that
$$
\liminf_{U\ni z\rightarrow z_0}(h^{\star}_{K,D}(z)-u(z))\geq 0,\quad z_0\in\partial U\cap D.
$$  
Then $u\leq h^{\star}_{K,D}$ in $U.$
\label{comparison}
\end{tw}
\begin{proof}
We know that $(dd^c h^{\star}_{K,D})^n=0$ on $D\setminus K.$ In particular, $(dd^ch^{\star}_{K,D})^n\leq(dd^cu)^n$ in $U.$ Moreover, $\lim_{z\rightarrow z_0}h^{\star}_{K,D}=1, z_0\in\partial D.$ Using Theorem \ref{domination} we get the conclusion.
\end{proof}
\begin{df}[see \cite{Sto},\cite{Ze}]\rm
Let $X$ be an irreducible Stein space. Then $X$ is called \emph{weakly parabolic} if there exists a plurisubharmonic continuous exhaustion function $g:X\rightarrow[0,\infty)$ such that $\log g$ is plurisubharmonic and satisfies $(dd^c\log g)^n=0$ on $X\setminus g^{-1}(0)$. 
\end{df}
\begin{tw}[see Theor{\`e}me 3.16 in \cite{Ze2}]
Let $X$ be an irreducible, locally irreducible weakly parabolic Stein space with some potential $g,$ let $K\s X$ be compact and let $U\s X$ be an open neighborhood of $\hat{K}_X.$ Then there exists a compact, holomorphically convex and locally $L$-regular (see \cite{Ze2}, Definition 3.13) set $E$ with $\hat{K}_X\s E\s U.$
\label{3.16}
\end{tw}
\section{Proof of the main result}
\label{sec4} 
\begin{proof}[Proof of Theorem \ref{pierscien}] The idea of the proof here is the approximation. First (Steps 1-4) we show that if we know the conclusion holds true for compact sets $A$ (and holomorphically convex, while we consider assumption (A)), then we are able to prove the theorem in its full generality. In Steps 5 and 6 we show that in fact we have the above mentioned property. The proof here is by approximation of $A$ from above by compacta (holomorphically convex, when we work with assumption (A)) with continuous relative extremal functions. The argument however must be more delicate than the one given in \cite{J3}, where such approximation do not require the holomorphic convexity, and additionally, is given just by the $\varepsilon$-envelopes of a set $A$.\\
Fix $0<r<s\leq 1$ and put 
$$\Delta[r]:=\{z\in D:h_{A,D}^{\star}(z)\leq r\},$$
$$L:=h^{\star}_{\Delta(r),\Delta(s)},\quad R:=\max\Big\{0,\frac{h^{\star}_{A,D}-r}{s-r}\Big\}.$$
Observe that $L\geq R.$ Thus we only need to prove the opposite inequality.
\indent\textit{Step 1.} We may assume that $s=1.$\\
The proof of Step 1 is the same for both assumptions, (A) and (B).
Take $0<r<s<1.$ Then $A\cap S$ is nonpluripolar for any connected component $S$ of $\Delta(s),$ and there is $h^{\star}_{A,\Delta(s)}=(1/s)h^{\star}_{A,D}$ on $\Delta(s)$ (this is because of Lemma \ref{sublevel} and the fact that for $D$ as in the assumptions, thanks to Josefson's theorem, we have that for any $P\in\mc{PLP}(D)$ there exists a $u\in\mc{PSH}(D),u\leq 0$ and nonconstant, such that $P\s\{u=-\infty\}$, from which follows that $h^{\star}_{A\cup P,D}=h^{\star}_{A,D}$ for any $A\s D$ and pluripolar set $P$ (see Theorem \ref{alehyane}). Finally, $A\setminus A^{\star}$ is pluripolar - see Lemma 2.6 from \cite{AH}). As a consequence, we get $L=h_{\Delta(r),\Delta(s)}=h_{\{h^{\star}_{A,\Delta(s)}<\frac{r}{s}\},\Delta(s)},R=\max\{0,\frac{h^{\star}_{A,\Delta(s)}-\frac{r}{s}}{1-\frac{r}{s}}\}.$ 
Thus, the problem for the data $(D,A,r,s)$ is done if only it is done for the data $(S,A\cap S,\frac{r}{s},1)$, where $S$ is as above.\\
\indent\textit{Step 2.} Approximation. Let $A_{\nu}\nearrow A, D_{\nu}\nearrow D$, where $A_{\nu}\s D_{\nu}$ is nonpluripolar for each $\nu\in\mb{N}.$ Then, if the conclusion holds true for the data $(D_{\nu},A_{\nu},r,1)$,\ $\nu\in\mb{N},$ then it holds true for $(D,A,r,1),$ as well.\\
Indeed, there is $h^{\star}_{A_{\nu},D_{\nu}}\searrow h^{\star}_{A,D}$ (by virtue of Proposition \ref{3.2.23}). Hence $\{h^{\star}_{A_{\nu},D_{\nu}}<r\}\nearrow\Delta(r)$ and $h^{\star}_{\{h^{\star}_{A_{\nu},D_{\nu}}<r\},D}\searrow h^{\star}_{\Delta(r),D}.$\\
\indent Using Step 1 and Step 2, from now on we assume that $A\s\s D$ and instead of $D$ we consider $\Omega_c(\psi),$ some sublevel set of a real analytic strongly plurisubharmonic exhaustion function of $D$.\\
\indent\textit{Step 3.} Assume that the assumption (B) is satisfied. Then, if the conclusion holds true for all nonpluripolar compact sets $A$, then it holds also for all nonpluripolar sets $A.$\\
Indeed, by Step 2, the conclusion holds for all non-empty open sets $A$. Take a nonpluripolar set $A$. Since the set $\Delta(\varepsilon)$ is open, we have
$$
h^{\star}_{\{h^{\star}_{\Delta(\varepsilon),D}<r\},D}=\max\Big\{0,\frac{h^{\star}_{\Delta(\varepsilon),D}-r}{1-r}\Big\},\quad\varepsilon\in(0,1).
$$
Then $\frac{h^{\star}_{A,D}-\varepsilon}{1-\varepsilon}\leq h^{\star}_{\Delta(\varepsilon),D}\leq h^{\star}_{A,D}$ (because of the Proposition \ref{3.2.15}), from which follows $h^{\star}_{\Delta(\varepsilon),D}\nearrow h^{\star}_{A,D}$ as $\varepsilon\searrow 0.$ Moreover, 
$$
\Big\{h^{\star}_{\Delta(\varepsilon),D}<\frac{r-\varepsilon}{1-\varepsilon}\Big\}\s\Delta(r)\s\{h^{\star}_{\Delta(\varepsilon),D}<r\},\quad\varepsilon\in(0,r), 
$$
which implies
$$
\max\Big\{0,\frac{h^{\star}_{\Delta(\varepsilon),D}-\frac{r-\varepsilon}{1-\varepsilon}}{1-\frac{r-\varepsilon}{1-\varepsilon}}\Big\}\geq h^{\star}_{\Delta(r),D}\geq\max\Big\{0,\frac{h^{\star}_{\Delta(\varepsilon),D}-r}{1-r}\Big\},\quad\varepsilon\in(0,r),
$$
and we get the conclusion as $\varepsilon\searrow 0.$\\
Thus the proof under assumptions of (B) reduces to the case where $A$ is compact.\\
\indent\textit{Step 4.} Assume that assumption (A) is satisfied. Then, if theorem holds true for all nonpluripolar compact and holomorphically convex sets $A,$ then it holds true for all nonpluripolar sets $A$.\\
Take a nonpluripolar set $A$. The set $\Delta(\varepsilon)$ is Runge in $D$ (and in particular it is a Stein space; see Proposition \ref{Runge}), so using approximation by compact holomorphically convex sets we see that the result holds true for the sets $A=\Delta(\varepsilon).$ We finish the proof of Step 4 as in the Step 3.\\
\indent\textit{Step 5.} The case where $A$ is compact and $h^{\star}_{A,D}$ is continuous.\\
The proof is parallel for both assumptions, (A) and (B).
The set $\Delta[r]$ is compact (by virtue of the continuity of $h^{\star}_{A,D}$ and Proposition \ref{4.5.2}). Let $u\in\mc{PSH}(D),u\leq 1,u\leq 0$ on $\Delta[r].$ Put $U:=D\setminus\Delta[r].$ Then for a $z_0\in\partial U$ we obtain
$$
\liminf_{U\ni z\rightarrow z_0}(h^{\star}_{A,D}(z)-(1-r)u(z)-r)\geq 0.
$$
Hence $(1-r)u+r\leq h^{\star}_{A,D}$ in $U$ (see Theorem \ref{comparison}). Thus $h_{\Delta[r],D}\leq R$ and $h^{\star}_{\Delta[r],D}\equiv R$. Finally, considering a sequence of positive numbers $(r_i)_{i\in\mb{N}}$ increasing to $r$ we get $L\equiv R.$\\
\indent\textit{Step 6.} The case where $A$ is compact.\\
First we carry out a construction of a decreasing sequence $(A_j)_{j\in\mb{N}}$ of closed sets containing $A,$ and being a finite unions of closed ``balls''.\\  
Since $D$ is metrizable (for both assumptions, (A) and (B), by virtue of Urysohn's Metrization Theorem), there exists a metric $d$, which gives the topology of $D$.\\
In the case where $D$ is a Stein space take a finite set of charts $(U_i,\varphi_i,B_i,V_i)$,\ $i=1,\ldots, s,$ and corresponding sets $\hat{\mb{B}}(a_i,r_i),$ such that $\hat{\mb{B}}(a_i,r_i)\s\s U_i$ and $\varphi_i:\hat{\mb{B}}(a_i,r_i)\rightarrow B_i\cap\mb{B}(\varphi_i(a_i),r_i)\s\s V_i$ is a biholomorphism, $i=1,\ldots, s$, satisfying $A\s\bigcup\limits_{i=1}^{s}\hat{\mb{B}}(a_i,r_i).$\\
We construct a set $A_1$. Fix an $a\in A.$ Without loss of generality we may assume that $a\in\hat{\mb{B}}(a_1,r_1)\s U_1.$ Take a number $r_a<1$ with $\mb{B}(\varphi_1(a),r_a)\s\mb{B}(\varphi_1(a_1),r_1)$ and small enough so that
$\hat{\mb{B}}(a,r_a)=\varphi^{-1}_1(B_1\cap\mb{B}(\varphi_1(a),r_a))\s\{x\in D:d(x,A)\leq 1\}.$ We may now choose a finite number of sets $\hat{\mb{B}}(a^1_l,r_{a^1_l}),l=1,\ldots,s_1$, with $a^1_l\in A,l=1,\ldots,s_1$,
and such that $A\s\bigcup\limits_{l=1}^{s_1}\hat{\mb{B}}(a^1_l,r_{a^1_l}),$
and define $A_1:=\bigcup\limits_{l=1}^{s_1}\overline{\hat{\mb{B}}(a^1_l,r_{a^1_l})}$.\\
Suppose we have constructed the set $A_j$ for some $j\in\mb{N}.$ Then we obtain  $A_{j+1}$ as follows: take an $a\in A$ and - as before - assume that $a\in\hat{\mb{B}}(a^j_1,r_{a^j_1})\s A_j\cap U_{i_a}$ for some $i_a\in\{1,\ldots,s\}.$ Take a number $r_a<\frac{1}{j+1}$ such that $\mb{B}(\varphi_{i_a}(a),r_a)\s\mb{B}(\varphi_{i_a}(a^j_1),r_{a^j_1})$ and small enough so that $\hat{\mb{B}}(a,r_a)=\varphi^{-1}_{i_a}(B_{i_a}\cap\mb{B}(\varphi_{i_a}(a),r_a))\s\{x\in D:d(x,A)\leq\frac{1}{j+1}\}.$ Choose a finite number of sets $\hat{\mb{B}}(a^{j+1}_m,r_{a^{j+1}_m}),m=1,\ldots,s_{j+1}$, with $a^{j+1}_m\in A,m=1,\ldots,s_{j+1}$, and such that $A\s\bigcup\limits_{m=1}^{s_{j+1}}\hat{\mb{B}}(a^{j+1}_m,r_{a^{j+1}_m}),$
and define $A_{j+1}:=\bigcup\limits_{m=1}^{s_{j+1}}\overline{\hat{\mb{B}}(a^{j+1}_m,r_{a^{j+1}_m})}$.\\ 
Clearly, $(A_j)_{j\in\mb{N}}$ is a decreasing sequence of compact sets being finite unions of closed ``balls'' with $\bigcap\limits_{j=1}^{\infty}A_j=A$.\\
In the subcase where $D$ is a manifold the above construction is carried out with $B_i=V_i.$\\
Two cases have to be considered.\\
\indent\textit{Case 1.} The case where (B) is satisfied.\\
Using Corollary 4.5.9 from \cite{K1} (which is also true for our context and our ``balls'', with a proof which goes along the same lines as in \cite{K1}: we only need to use the approximation of $D$ by strongly pseudoconvex domains and Theorem 10.4 from \cite{Sad} instead of Proposition 4.5.3, and pass to $\mb{C}^n$ by charts) we see that $h_{A_j,D}=h^{\star}_{A_j,D}$ is continuous. Then we have $$h_{\{h_{A_j,D}\leq r\},D}=\max\Big\{0,\frac{h_{A_j,D}-r}{1-r}\Big\}.$$ Also, $h_{A_j,D}\nearrow h_{A,D}$ as $j\nearrow\infty$ (in view of Proposition \ref{3.2.24}). Hence $\{h_{A_j,D}\leq r\}\searrow\{h_{A,D}\leq r\}$ as $j\nearrow\infty.$ Thus $h_{\{h_{A_j,D}\leq r\},D}\nearrow h_{\{h_{A,D}\leq r\},D},$ from which follows $$h_{\{h_{A,D}\leq r\},D}=\max\Big\{0,\frac{h_{A,D}-r}{1-r}\Big\}\leq R.$$ Hence $h^{\star}_{\{h_{A,D}\leq r\},D}\leq R.$ Since the set $\{h_{A,D}\leq r\}\setminus\Delta[r]$ is pluripolar, $h^{\star}_{\Delta[r],D}\leq R$ and, as in Step 5, $L\equiv R.$\\
\indent\textit{Case 2.} The case where (A) is satisfied and $A$ is additionally holomorphically convex.\\
Here we do not know if the relative extremal functions of $A_j$'s are continuous. However, we may once again use the approximation argument to shift the situation to the case of Step 5. It is to do as follows:\\
For any $j\in\mb{N}$ put $U_j:=\bigcup\limits_{m=1}^{s_{j}}\hat{\mb{B}}(a^{j}_m,r_{a^{j}_m})$.
Observe that the sequence $(U_j)_{j\in\mb{N}}$ of open sets is decreasing and enjoys property that for any open set $U$ containing $A$ there is an index $j(U)$ with $U_j\s U$ for all $j\geq j(U).$\\
We use now Theorem \ref{3.16} for $U_j$'s as follows: for $U_1,$ using the same method as in the proof of Theorem \ref{3.16} (given in \cite{Ze2}), we find a compact and holomorphically convex set $E_1$ with continuous relative extremal function and such that $A\s{\rm int}E_1\s E_1\s U_1$ (it suffices to consider $\delta+\varepsilon$ with small $\varepsilon$, instead of $\delta$ in the definition of $E$ in the proof in \cite{Ze2}) Suppose we have found sets $E_1,\ldots, E_j$ for some $j\in \mb{N}.$ In this situation we obtain $E_{j+1}$ using the argument given above for $U_{j+1}\cap{\rm int}E_j$ instead of $U_1.$ We easily see that the decreasing sequence of sets $(E_j)_{j\in\mb{N}}$ gives an approximation of $A$ from above by holomorphically convex compacta with continuous relative extremal functions. It now suffices to use the same argument as in the end of the Case 1.
\end{proof} 
\section{Applications of the main result}
\label{sec5}
In this section we give some applications of our main result. First we need to define the generalized $(N,k)$-crosses in the context of Stein manifolds. 
Let $D_j$ be an $n_j$-dimensional Stein manifold and let $\varnothing\neq A_j\s D_j$ for $j=1,\ldots,N,$ $N\geq 2.$ For $k\in\{1,\ldots,N\}$ let $I(N,k):=\{\alpha=(\alpha_1,\ldots,\alpha_N)\in\{0,1\}^N:|\alpha|=k\},$ where  $|\alpha|:=\alpha_1+\ldots+\alpha_N.$ Put 
\begin{displaymath}
\mathbf{\mc{X}}_{\alpha,j}:=\begin{cases}
D_j, &\text{ if }\alpha_j=1\\
A_j, &\text{ if }\alpha_j=0
\end{cases},\quad \mc{X}_{\alpha}:=\prod_{j=1}^N \mc{X}_{\alpha,j}.
\end{displaymath}
\indent For $\alpha\in I(N,k)$ such that $\alpha_{r_1}=\ldots=\alpha_{r_{k}}=1,\alpha_{i_1}=\ldots=\alpha_{i_{N-k}}=0,$  where $r_1<\ldots<r_k$ and $i_1<\ldots<i_{N-k}$, put
\begin{displaymath}
D_{\alpha}:=\prod_{s=1}^k D_{r_s},\quad A_{\alpha}:=\prod_{s=1}^{N-k} A_{i_s}.
\end{displaymath}
For an $a=(a_1,\ldots,a_N)\in\mc{X}_{\alpha},$ $\alpha$ as above, put $a_{\alpha}^0:=(a_{i_1},\ldots,a_{i_{N-k}})\in A_{\alpha}.$ Analogously, define 
$a_{\alpha}^1:=(a_{r_1},\ldots,a_{r_k})\in D_{\alpha}.$ For every $\alpha\in I(N,k)$ and every $a=(a_{i_1},\ldots,a_{i_{N-k}})\in A_{\alpha}$ define 
\begin{displaymath}
\boldsymbol{i}_{a,\alpha}=(\boldsymbol{i}_{a,\alpha,1},\ldots,\boldsymbol{i}_{a,\alpha,N}):D_{\alpha}\rightarrow\mc{X}_{\alpha},
\end{displaymath}
\begin{displaymath}
\boldsymbol{i}_{a,\alpha,j}(z):=\begin{cases}
z_j, &\text{if }\alpha_j=1\\
a_{j}, &\text{if }\alpha_j=0
\end{cases},\quad j=1,\ldots, N,\quad z=(z_{r_1},\ldots,z_{r_k})\in D_{\alpha}
\end{displaymath}
(if $\alpha_j=0,$ then $j\in\{i_1,\ldots,i_{N-k}\}$ and if $\alpha_j=1,$ then $j\in\{r_1,\ldots,r_{k}\}$).
Similarly, for any $\alpha\in I(N,k)$ and any $b=(b_{r_1},\ldots,b_{r_k})\in D_{\alpha}$ define 
\begin{displaymath}
\boldsymbol{l}_{b,\alpha}=(\boldsymbol{l}_{b,\alpha,1},\ldots,\boldsymbol{l}_{b,\alpha,N}):A_{\alpha}\rightarrow\mc{X}_{\alpha},
\end{displaymath}
\begin{displaymath}
\boldsymbol{l}_{b,\alpha,j}(z):=\begin{cases}
z_j, &\text{if }\alpha_j=0\\
b_{j}, &\text{if }\alpha_j=1
\end{cases},\quad j=1,\ldots, N,\quad z=(z_{i_1},\ldots,z_{i_{N-k}})\in A_{\alpha}. 
\end{displaymath}
\begin{df} (cf. \cite{L1})
For any $\alpha\in I(N,k)$ let $\Sigma_{\alpha}\s A_{\alpha}.$ We define a~\emph{generalized $(N,k)$-cross}
\begin{displaymath}
\mathbf{T}_{N,k}:=\mb{T}_{N,k}((A_j,D_j)_{j=1}^N,(\Sigma_{\alpha})_{\alpha\in I(N,k)})=\bigcup_{\alpha\in I(N,k)}
\{a\in\mc{X}_{\alpha}:a_{\alpha}^0\notin\Sigma_{\alpha}\}
\end{displaymath}
and its \emph{center}
\begin{displaymath}
\mf{C}(\mathbf{T}_{N,k}):=\mathbf{T}_{N,k}\cap(A_1\times\ldots\times A_N).
\end{displaymath}
\end{df} 
It is straightforward that
\begin{displaymath}
\mf{C}(\mathbf{T}_{N,k})=(A_1\times\ldots\times A_N)\setminus\bigcap_{\alpha\in I(N,k)}\{z\in A_1\times\ldots\times A_N:z_{\alpha}^0\in\Sigma_{\alpha}\},
\end{displaymath}
which implies that $\mf{C}(\mathbf{T}_{N,k})$ is non-pluripolar provided that $A_1\times\ldots\times A_N$ is non-pluripolar and at least one of the $\Sigma_{\alpha}$'s is pluripolar (cf. Proposition 2.3.31 from \cite{J2}).
Note that if we take $\Sigma_{\alpha}=\varnothing$ for every $\alpha~\in~I(N,k),$ then in the definition above we get the \textit{$(N,k)$-cross} (see~\cite{J3}) 
\begin{displaymath}
\mathbf{X}_{N,k}=\mb{X}_{N,k}((A_j,D_j)_{j=1}^N):=\mb{T}_{N,k}((A_j,D_j)_{j=1}^N,(\varnothing)_{\alpha\in I(N,k)}). 
\end{displaymath}
\begin{df}[\cite{J3}]
For an $(N,k)$-cross define its \emph{envelope} by
\begin{multline*}
\hat{\mathbf{X}}_{N,k}=\hat{\mb{X}}_{N,k}((A_j,D_j)_{j=1}^N)\\:=\Big\{(z_1,\ldots,z_N)\in D_1\times\ldots\times D_N:\sum_{j=1}^N h^{\star}_{A_j,D_j}(z_j)<k\Big\}.
\end{multline*}
Note the obvious inclusion $\hat{\mathbf{X}}_{N,k-1}\s\hat{\mathbf{X}}_{N,k}$.
\end{df}
As it was already mentioned, using Theorem \ref{pierscien} we may derive a formula for the relatively extremal function of the envelope of $(N,k-1)$-cross with respect to the envelope of $(N,k)$-cross, which will play a fundamental role in the proof of Theorem \ref{main}.
\begin{tw}
Let $D_j$ be a Stein manifold and let $A_j\s D_j$ be locally pluriregular, $j=1,\ldots,N.$ Then
\begin{displaymath}
h^{\star}_{\hat{\mathbf{X}}_{N,k-1},\hat{\mathbf{X}}_{N,k}}(z)=\max\Big\{0,\sum_{j=1}^N h^{\star}_{A_j,D_j}(z_j)-k+1\Big\},\quad z=(z_1,\ldots,z_N)\in\hat{\mathbf{X}}_{N,k}.
\end{displaymath} 
\label{Obwiednia w obwiedni}
\end{tw}
\begin{proof}
We carry out this proof exactly the same as in \cite{J3}, bearing in mind that the product property for relatively extremal function is true also for domains in Stein manifolds (see \cite{EP1}).
\end{proof}
\begin{df}
We say that a function $f:\mathbf{T}_{N,k}\rightarrow\mb{C}$ is \emph{separately holomorphic on $\mathbf{T}_{N,k}$} if for every $\alpha\in I(N,k)$ and for every $a\in A_{\alpha}\setminus\Sigma_{\alpha}$ the function
\begin{displaymath}
D_{\alpha}\ni z\mapsto f(\boldsymbol{i}_{a,\alpha}(z))
\end{displaymath} is holomorphic. In this case we write $f\in\mc{O}_s(\mathbf{T}_{N,k}).$\\
\indent We denote by $\mc{O}_s^c(\mathbf{T}_{N,k})$ the space of all $f\in\mc{O}_s(\mathbf{T}_{N,k})$ such that for any $\alpha\in I(N,k)$ and for every $b\in D_{\alpha}$ the function
\begin{displaymath}
A_{\alpha}\setminus\Sigma_{\alpha}\ni z\mapsto f(\boldsymbol{l}_{b,\alpha}(z))
\end{displaymath} is continuous.
\end{df}
\begin{tw}
[cf. Theorem 7.1.4 in \cite{J2}]
Let $D_j$ be a Stein manifold, let $A_j\s D_j$ be locally pluriregular, $j=1,\ldots,N$. Let $\Sigma_{\alpha}\s A_{\alpha}$ be pluripolar, $\alpha\in I(N,1)$. Put~$\mathbf{X}_{N,1}:=\mb{X}_{N,1}((A_j,D_j)_{j=1}^N), \mathbf{T}_{N,1}:=\mb{T}_{N,1}((A_j,D_j,\Sigma_j)_{j=1}^N).$ Let $f\in\mc{O}_s^c(\mathbf{X}_{N,1}).$ Then there exists a uniquely determined $\hat{f}\in\mc{O}(\hat{\mathbf{X}}_{N,1})$ such that $\hat{f}=f$ on $\mathbf{T}_{N,1}$ and $\hat{f}(\hat{\mathbf{X}}_{N,1})\s f(\mathbf{T}_{N,1}).$ 
\label{Main cross theorem}
\end{tw}
\begin{proof}
The proof may be rewritten almost verbatim from \cite{J2}.
\end{proof}
\begin{tw}
\label{main}
Let $D_j$ be a Stein manifold and $A_j\s D_j$ be locally pluriregular, $j=1,\ldots, N.$ Take $\Sigma_{\alpha}\s A_{\alpha}$ pluripolar, $\alpha\in I(N,k)$ and put $\mathbf{T}_{N,k}:=\mb{T}_{N,k}((A_j,D_j)_{j=1}^N,(\Sigma_{\alpha})_{\alpha\in I(N,k)}), \mathbf{X}_{N,k}:=\mb{X}_{N,k}((A_j,D_j)_{j=1}^N).$ Then any function $f\in\mc{F}:=\mc{O}_{s}^{c}(\mathbf{T}_{N,k})$ admits a holomorphic extension $\hat{f}\in\mc{O}(\hat{\mathbf{X}}_{N,k})$ such that $\hat{f}=f$ on $\mathbf{T}_{N,k}$ and $\hat{f}(\hat{\mathbf{X}}_{N,k})\s f(\mathbf{T}_{N,k}).$
\end{tw}
\begin{proof} 
The inclusion $\hat{f}(\hat{\mathbf{X}}_{N,k})\s f(\mathbf{T}_{N,k})$ for $f\in\mc{F}$ is to obtain in a standard way (cf. Lemma 2.1.14 in \cite{J2}; observe it is also true in our context).\\
For each $D_j$ we may find an exhausting sequence of strongly pseudoconvex relatively compact open sets with smooth boundaries (by considering sublevel sets of a smooth strictly plurisubharmonic exhaustion function for each $D_j$).
Thus, it is enough to prove the theorem with additional assumptions that each $D_j$ is strongly pseudoconvex relatively compact open subset (with smooth boundary) of some Stein manifold $\tilde{D}_j$ and $A_j\s\s D_j.$\\
We apply induction over $N.$ There is nothing to prove in the case $N=k.$
Moreover, the case $k=1$ is solved by Theorem \ref{Main cross theorem}. Thus, the conclusion holds true for $N=2.$ Suppose it holds true for $N-1\geq 2.$ Now, we apply induction over $k.$ For $k=1,$ as mentioned, the result is known. Suppose that the conclusion is true for $k-1$ with $2\leq k\leq N-1.$\\
\indent Fix an $f\in\mc{F}.$ Define
\begin{displaymath}
Q:=Q_N=\{z_N\in A_N:\exists \alpha\in I_0(N,k): (\Sigma_{\alpha})_{(\cdot,z_N)}\notin\mc{PLP}\},
\end{displaymath}
where $I_0(N,k):=I(N,k)\cap\{\alpha:\alpha_N=0\}.$ Then $Q\in\mc{PLP}$ (cf. Proposition 2.3.31 from \cite{J2}). For a $z_N\in A_N\setminus Q$ put
\begin{displaymath}
\mathbf{T}_{N-1,k}(z_N):=\mb{T}_{N-1,k}((A_j,D_j)_{j=1}^{N-1},((\Sigma_{(\beta,0)})_{(\cdot,z_N)})_{\beta\in I(N-1,k)}).
\end{displaymath}
Consider also the generalized $(N-1,k-1$-cross
\begin{displaymath}
\mathbf{T}_{N-1,k-1}:=\mb{T}_{N-1,k-1}((A_j,D_j)_{j=1}^{N-1},(\Sigma_{(\beta,1)})_{\beta\in I(N-1,k-1)}).
\end{displaymath}
It can be easily seen that for a fixed $z_N\in A_N\setminus Q$ we have
\begin{displaymath}
(\mathbf{T}_{N,k})_{(\cdot,z_N)}=\mathbf{T}_{N-1,k}(z_N)\cup\mathbf{T}_{N-1,k-1},
\end{displaymath}
where $(\mathbf{T}_{N,k})_{(\cdot,z_N)}$ is the fiber of the set $\mathbf{T}_{N,k}$ over $z_N.$ Define
\begin{displaymath}
\mathbf{Y}_{N-1,k}:=\mb{X}_{N-1,k}((A_j,D_j)_{j=1}^{N-1}),\quad
\mathbf{Y}_{N-1,k-1}:=\mb{X}_{N-1,k-1}((A_j,D_j)_{j=1}^{N-1}).
\end{displaymath}
For any $z_N\in A_N\setminus Q$ we have $f(\cdot,z_N)\in\mc{O}_s^c(\mathbf{T}_{N-1,k}(z_N))$ and, moreover, for any $z_N\in D_N$ we have $f(\cdot,z_N)\in\mc{O}_s^c(\mathbf{T}_{N-1,k-1}).$
Then, by inductive assumption, for any $z_N\in A_N\setminus Q$ there exists an $\hat{f}_{z_N}\in\mc{O}(\hat{\mathbf{Y}}_{N-1,k})$ such that $\hat{f}_{z_N}=f(\cdot,z_N)$ on
$\mathbf{T}_{N-1,k}(z_N).$ Analogously, for any $z_N\in D_N$
there exists a $\hat{g}_{z_N}\in\mc{O}(\hat{\mathbf{Y}}_{N-1,k-1})$ such that $\hat{g}_{z_N}=f(\cdot,z_N)$ on
$\mathbf{T}_{N-1,k-1}.$\\
\indent Define a $2-$fold classical cross (cf. \cite{J6})
\begin{displaymath}
\mathbf{Z}:=\mb{X}_{2,1}((B_j,E_j)_{j=1}^2),
\end{displaymath}
where $B_1=\hat{\mathbf{Y}}_{N-1,k-1},B_2=A_N\setminus Q,E_1=\hat{\mathbf{Y}}_{N-1,k},E_2=D_N.$ Clearly 
\begin{displaymath}
\mathbf{Z}=(\hat{\mathbf{Y}}_{N-1,k-1}\times D_N)\cup(\hat{\mathbf{Y}}_{N-1,k}\times (A_N\setminus Q)).
\end{displaymath}
Applying Lemma \ref{Obwiednia w obwiedni} and pluripolarity of $Q$ we get $\hat{\mathbf Z}=\hat{\mathbf{X}}_{N,k}.$\\
\indent Let $F:\mathbf{Z}\rightarrow \mb{C}$ be given by the formula
\begin{displaymath}
F(z',z_N):=\begin{cases}
\hat{f}_{z_N}(z'), &\text{if } (z',z_N)\in\hat{\mathbf{Y}}_{N-1,k}\times(A_N\setminus Q),\\
\hat{g}_{z_N}(z'), &\text{if } (z',z_N)\in\hat{\mathbf{Y}}_{N-1,k-1}\times D_N.
\end{cases}
\end{displaymath}
\indent First, observe that $F$ is well-defined. Indeed, we only have to check that for any $z_N\in A_N\setminus Q$ we have equality $\hat{f}_{z_N}=\hat{g}_{z_N}$ on $\hat{\mathbf{Y}}_{N-1,k-1}.$ In fact, since both $\hat{f}_{z_N}$ and $\hat{g}_{z_N}$ are extensions of $f(\cdot,z_N),$ we only need to prove existence of some non-pluripolar set $B\s\mathbf{T}_{N-1,k}(z_N)\cap\mathbf{T}_{N-1,k-1}$ and use the identity principle. Observe that the set
\begin{displaymath}
B:=\mf{C}(\mathbf{T}_{N-1,k}(z_N))\cap\mf{C}(\mathbf{T}_{N-1,k-1})
\end{displaymath}
is good for our purpose.\\
\indent Now we prove that $F\in\mc{O}_s(\mathbf{Z}).$
We have to prove that for each $z'\in\hat{\mathbf{Y}}_{N-1,k-1}$ the function $D_N\in z_N\mapsto F(z',z_N)$ is holomorphic (or equivalently, that $F\in\mc{O}(\hat{\mathbf{Y}}_{N-1,k-1}\times D_N)).$ We already know that $F(\cdot,z_N)$ is holomorphic for every $z_N\in D_N.$
To show that $F\in\mc{O}(\hat{\mathbf{Y}}_{N-1,k-1}\times D_N)$ we will use Terada's theorem (or the Cross theorem for manifolds - see \cite{J2}, Theorem 6.2.2). Put
\begin{displaymath}
\mathbf{W}_{N-1,k-1}:=\mb{T}_{N-1,k-1}((A_j,D_j)_{j=2}^N,(\Sigma_{(1,\beta)})_{\beta\in I(N-1,k-1)}),
\end{displaymath}
\begin{displaymath}
\mathbf{Z}_{N-1,k-1}:=\mb{X}_{N-1,k-1}((A_j,D_j)^N_{j=2}).
\end{displaymath}
\indent From the inductive assumption, for any $z_1\in D_1$ there exists an $\hat{h}_{z_1}\in\mc{O}(\hat{\mathbf{Z}}_{N-1,k-1})$ with $\hat{h}_{z_1}=f(z_1,\cdot)$ on $\mathbf{W}_{N-1,k-1}.$ Thus we get
\begin{displaymath}
F(z_1,\ldots,z_N)=f(z_1,\ldots,z_N)=\hat{h}_{z_1}(z_2,\ldots,z_N) 
\end{displaymath} 
for $(z_1,\ldots,z_N)\in(\mathbf{T}_{N-1,k-1}\times D_N)\cap(D_1\times\mathbf{W}_{N-1,k-1}).$ It suffices to show that there exists a non-pluripolar set $C$ such that 
$$C\times D_N\s(\mathbf{T}_{N-1,k-1}\times D_N)\cap(D_1\times\mathbf{W}_{N-1,k-1}).$$ 
It is easy to see that the set with the required properties is
\begin{displaymath}
C:=\mf{C}(\mathbf{T}_{N-1,k-1})\setminus\bigcap_{\alpha\in I(N,k):\alpha_1=\alpha_N=1}\{z\in\prod_{j=1}^{N-1}A_j:
z_{\alpha}^0\in\Sigma_{\alpha}\}.
\end{displaymath}
From the Cross theorem for manifolds we get the existence of a function $\hat{f}\in\mc{O}(\hat{\mathbf{Z}})$ with $\hat{f}=F$ on $\mathbf{Z}.$\\
\indent We have to verify that $\hat{f}=f$ on $\mathbf{T}_{N,k}.$
Take a point $a\in\mathbf{T}_{N,k}.$ The conclusion is obvious if $a\in\mathbf{T}_{N-1,k-1}\times D_N\s\mathbf{Z}.$ Suppose, without losing generality, that $a=~(a_1,\ldots,a_k,a_{k+1},\ldots,a_N)\in D_1\times\ldots\times D_k\times(A_{\alpha}\setminus\Sigma_{\alpha}),$ where $\alpha=(\underbrace{1,\ldots,1}_{k},\underbrace{0,\ldots,0}_{N-k}).$ Observe that we have
\begin{displaymath}
\mc{T}:=\bigcup_{z_N\in A_N\setminus Q}\mathbf{T}_{N-1,k}(z_N)\times\{z_N\}\s\hat{\mathbf{Y}}_{N-1,k}\times(A_N\setminus Q)\s\mathbf{Z}.
\end{displaymath}
We have also
\begin{displaymath}
\mc{T}\s\bigcup_{z_N\in A_N\setminus Q}(\mathbf{T}_{N,k})_{(\cdot,z_N)}\times\{z_N\}\s\mathbf{T}_{N,k}.
\end{displaymath}
\indent Thus, if $b=(b',b_N)\in\mc{T},$ then $\hat{f}(b)=F(b)=\hat{f}_{b_N}(b')=f(b).$ Bearing this in mind, we easily see that it suffices to find a sequence $$(b^{\nu})_{\nu=1}^{\infty}\s\mc{T}\cap\{(a_1,\ldots,a_k)\}\times (A_{\alpha}\setminus\Sigma_{\alpha})$$ 
such that $b^{\nu}\rightarrow a,$ and then continuity of $f(a_1,\ldots,a_k,\cdot)$ will end the proof.\\
\indent Since $Q$ is pluripolar, there exists a sequence $({b_N}^{\nu})$ convergent to $a_N$ such that $({b_N}^{\nu})~\s~ A_N\setminus Q.$ Put $P:=\bigcup\limits_{\nu=1}^{\infty}(\Sigma_{\alpha})_{(\cdot,{b_N}^{\nu})},$ which is a pluripolar set. This guarantees the existence of a sequence $(({b_{k+1}}^{\nu},\ldots,{b_{N-1}}^{\nu}))\s (A_{k+1}\times\ldots\times A_{N-1})\setminus P,$ convergent to $(a_{k+1},\ldots,a_{N-1}).$ Finally we put $b^{\nu}:=(a_{\alpha}^1,{b_{k+1}}^{\nu},\ldots,{b_{N-1}}^{\nu}).$ It is obvious that $b^{\nu}\rightarrow a$ and that for every $\nu\in\mb{N}, b^{\nu}\in\mathbf{T}_{N-1,k}({b_N}^{\nu})\times\{{b_N}^{\nu}\}\s\mc{T}.$
\end{proof}
\indent It is well known that the limit of an increasing sequence of Stein manifolds need not to be Stein (see, for example, \cite{F1}). Observe that our proofs work also for such objects, from which follows that theorems \ref{pierscien} and \ref{Obwiednia w obwiedni} hold true in more general context than Stein manifolds. It is however an open problem whether them hold true for arbitrary complex manifolds or spaces. Also, we do not know whether Theorem \ref{Obwiednia w obwiedni} can be extended to the context of at least Stein spaces.


\begin{thebibliography}{HD}

\normalsize
\baselineskip=15pt

\bibitem[AH]{AH}O.~Alehyane,~J.-M.~Hecart,
\emph{Propri{\'{e}}t{\'{e}} de stabilit{\'{e}} de la fonction extr{\'{e}}male relative},
Potential Anal. 21 (2004), 363-373.
\bibitem[AZ]{A1}O.~Alehyane,~A.~Zeriahi,
\emph{Une nouvelle version du th{\'{e}}or{\`e}me d'extension de Hartogs pour les applications s{\'{e}}par{\'{e}}ment holomorphes
entre espaces analytiques}, 
Ann. Polon.Math. 76 (2001), 245--278.
\bibitem[B]{B1}E.~Bedford,
\emph{The operator $(dd^c)^n$ on complex spaces},
S{\'{e}}minaire d'Analyse Lelong-Skoda et Colloque Wimereux, Lecture Notes in Math. 919, Springer-Verlag, Berlin 1981.
\bibitem[BT]{BT1}E.~Bedford,~B.A.~Taylor,
\emph{A new capacity for plurisubharmonic functions}
Acta Math., 149 (1982), 1-40.
\bibitem[EP]{EP1}A.~Edigarian,~E.A.~Poletsky,
\emph{Product property of the relative extremal function},
Bull. Sci. Acad. Polon. 45 (1997), 331-335.
\bibitem[F]{F1}J.E.~Fornaess,
\emph{An increasing sequence of Stein manifolds whose limit is not Stein},
Math. Ann. 223 (1976), 275-277.
\bibitem[FN]{FN1}J.E.~Fornaess,~R.~Narasimhan,
\emph{The Levi problem on complex spaces with singularities},
Math. Ann. 248 (1980), 47-72. 
\bibitem[GR]{GR}R.C.~Gunning,~H.~Rossi,
\emph{Analytic functions of several complex variables},
Prentice-Hall, Englewood Cloffs, NJ, 1965.
\bibitem[JP1]{J6}M.~Jarnicki,~P.~Pflug,
\emph{An extension theorem for separately holomorphic functions with pluripolar singularities}, 
Trans. Amer. Math. Soc. 355 (2003), 1251--1267.
\bibitem[JP2]{J3}M.~Jarnicki,~P.~Pflug,
\emph{A new cross theorem for separately holomorphic functions}, 
Proc. Amer. Math. Soc. 138 (2010), 3923--3932.
\bibitem[JP3]{J2}M.~Jarnicki,~P.~Pflug,
\emph{Separately analytic functions}, 
EMS Publishing House, 2011.
\bibitem[K]{K1}M.~Klimek,
\emph{Pluripotential theory}, London Math. Soc. Monogr. (N.S.) 6, Oxford University Press, New York, 1991.
\bibitem[L]{L1}A.~Lewandowski,
\emph{An extension theorem with analytic singularities for generalized $(N,k)$-crosses}, Ann. Polon. Math. 103 (2012), 193-208.
\bibitem[Lo]{Lo}S.~{\L}ojasiewicz,
\emph{Introduction to complex analytic geometry}, 
Translated from the Polish by Maciej Klimek, Birkh\"{a}user Verlag, 1991.
\bibitem[N1]{N1}R.~Narasimhan,
\emph{The Levi problem for complex spaces},
Math. Ann. 142 (1961), 355-365.
\bibitem[N2]{N2}R.~Narasimhan,
\emph{The Levi problem for complex spaces II},
Math. Ann. 146 (1962), 195-216.
\bibitem[S]{Sad}A.~Sadullaev,
\emph{Plurisubharmonic measures and capacities on complex manifolds},
Russian Math. Surveys 36 (1981), 61-119.
\bibitem[Sm]{Sm}P.A.N.~Smith,
\emph{Smoothing plurisubharmonic functions on complex spaces},
Math. Ann. 273 (1986), 397-413.
\bibitem[St]{Sto}W.~Stoll,
\emph{The characterization of strictly parabolic spaces},
Compositio Math. 44, no 1-3 (1981), 305-373.
\bibitem[Z1]{Ze}A.~Zeriahi,
\emph{Fonctions plurisousharmoniques extr{\'{e}}males. Approximation et croissance des fonctions holomorphes sur des ensembles alg{\'{e}}briques} The{\`s}e de Doctorat d'{\'{E}}tat, Sciences,
U.P.S. Toulouse, 1986.
\bibitem[Z2]{Ze2}A.~Zeriahi,
\emph{Fonction de Green pluricomplexe {\`a} p{\^{o}}le {\`a} l'infini sur un espace de Stein parabolique et applications},
Math. Scand. 69 (1991), 89-126.
\end{thebibliography}
\end{document}